\documentclass[a4paper,12pt]{amsart}
\usepackage{latexsym,amscd,amssymb,url}
\usepackage{xcolor}
\definecolor{dblue}{rgb}{0,0,.6}
\usepackage[colorlinks=true, linkcolor=dblue, citecolor=dblue, filecolor = dblue, menucolor = dblue, urlcolor = dblue]{hyperref}

\usepackage[all,cmtip]{xy}  
\usepackage{graphicx}
\usepackage{color}
\usepackage{hyperref}
\usepackage{enumerate}

\newtheorem{theorem}{Theorem}

\theoremstyle{plain}

\newtheorem{corollary}[theorem]{Corollary}

\newtheorem{lemma}[theorem]{Lemma}

\newtheorem{proposition}[theorem]{Proposition}
\newtheorem{remark}[theorem]{Remark}

\newcommand{\del}{\partial}
\newcommand{\Z}{\mathbb Z}

\newcommand{\C}{\mathbb C}

\newcommand{\CP}{\mathbb P}

\newcommand{\Spec}{\operatorname{Spec}}

\newcommand{\CH}{\operatorname{CH}}

\newcommand{\discr}{\operatorname{discr}} 
\newcommand{\cl}{\operatorname{cl}} 
\newcommand{\ram}{\operatorname{ram}} 
\newcommand{\Frac}{\operatorname{Frac}}

\newcommand{\dashedlongrightarrow}{\xymatrix@1@=15pt{\ar@{-->}[r]&}}
\renewcommand{\longrightarrow}{\xymatrix@1@=15pt{\ar[r]&}}
\renewcommand{\mapsto}{\xymatrix@1@=15pt{\ar@{|->}[r]&}}
\renewcommand{\twoheadrightarrow}{\xymatrix@1@=15pt{\ar@{->>}[r]&}}
\newcommand{\hooklongrightarrow}{\xymatrix@1@=15pt{\ar@{^(->}[r]&}}
\newcommand{\congpf}{\xymatrix@1@=15pt{\ar[r]^-\sim&}}
\renewcommand{\cong}{\simeq}

\begin{document}  
\title[Quadric surface bundles over surfaces and stable rationality]{Quadric surface bundles over surfaces and stable rationality} 

\author{Stefan Schreieder}

\address{Mathematisches Institut, LMU M\"unchen, Theresienstr. 39, 80333 M\"unchen, Germany}
\email{schreieder@math.lmu.de}

\date{October 10, 2017; \copyright{\ Stefan Schreieder 2017}}
%


\begin{abstract}  
We prove a general specialization theorem which implies stable irrationality for a wide class of quadric surface bundles over rational surfaces.  
As an application, we solve, with the exception of two cases, the stable rationality problem for any very general complex projective quadric surface bundle over $\CP^2$, 
given by a symmetric matrix of homogeneous polynomials.  
Both exceptions degenerate over a plane sextic curve and the corresponding double cover is a K3 surface.
\end{abstract}

\maketitle

\section{Introduction}

Recently, Hassett, Pirutka and Tschinkel \cite{HPT,HPT2,HPT3} found the first three examples of families of quadric surface bundles over $\CP^2$,  
where the very general member is not stably rational.  
In each case, the degeneration locus is a plane octic curve. 
Smooth quadric surface bundles over rational surfaces typically deform to smooth bundles with a section, hence to smooth rational fourfolds.  
This allowed them to produce the first examples of smooth non-rational varieties that deform to rational ones.

In \cite{Sch17}, we introduced a variant of the method of Voisin \cite{voisin} and Colliot-Th\'el\`ene--Pirutka \cite{CT-Pirutka}, which allowed us to disprove stable rationality via a degeneration argument where a universally $\CH_0$-trivial resolution of the special fibre is not needed.  
The purpose of this paper is to show that one can use this technique to simplify the arguments in \cite{HPT,HPT2,HPT3} and to 
apply them to large classes of quadric surface bundles. 

The main result is the following general specialization theorem without resolutions; see Section \ref{subsec:conventions} below for what it means that a variety specializes to another variety.

\begin{theorem} \label{thm:quadric}
Let $X$ and $Y$ be complex projective varieties of dimension four.
Suppose that $X$ specializes to $Y$ and that there is a morphism $f:Y\longrightarrow S$ to a rational surface $S$, such that: 
\begin{enumerate} 
\item the generic fibre of $f$ is a smooth quadric surface $Q$ over $K=\C(S)$; \label{item:Q}
\item the discriminant $d\in K^\ast \slash (K^\ast)^2$ of $Q$ is nontrivial; \label{item:d}
\item $H^2_{nr}(\C(Y)\slash \C,\Z\slash 2)\neq 0$ 
\label{item:H_nr}
\end{enumerate}
Then $X$ is not stably rational.
\end{theorem}

Since $H^2_{nr}(\C(Y)\slash \C,\Z\slash 2)=H^2_{nr}(K(Y)\slash \C,\Z\slash 2)$, the assumptions in the above theorem concern only the generic fibre of $f$.
In particular, $f$ need not be flat and there is no assumption on the singularities of $Y$ at points which do not dominate $S$.
A universally $\CH_0$-trivial resolution of $Y$ is not needed.    
For a more general version which works also if the discriminant of $Q$ is possibly trivial, $X$ and $Y$ have arbitrary dimension and the generic fibre of $f$ is only stably birational to $Q$, see Theorem \ref{thm:quadric:2} below.   

The second unramified cohomology group in item (\ref{item:H_nr}) coincides with the $2$-torsion subgroup of the Brauer group of any resolution of singularities of $Y$.
Pirutka computed this group explicitly for any quadric surface over $\C(\CP^2)$ which satisfies (\ref{item:d}), see \cite[Theorem 3.17]{Pirutka}.
This gives rise to many examples to which the above theorem applies.
In this paper we will apply it only to a single example of Hassett, Pirutka and Tschinkel \cite[Proposition 11]{HPT}.

The proof of Theorem \ref{thm:quadric} uses Pirutka's results \cite{Pirutka} on the unramified cohomology of quadric surfaces over $\C(\CP^2)$, together with our aforementioned method from \cite{Sch17}, which builds on \cite{voisin} and \cite{CT-Pirutka}.

To give an application of Theorem \ref{thm:quadric}, let us consider a generically non-degenerate line bundle valued quadratic form $q:\mathcal E \longrightarrow \mathcal O_{\CP^2}(n)$, where $\mathcal E=\bigoplus_{i=0}^3 \mathcal O_{\CP^2}(-r_i)$ is split and such that the quadratic form $q_s$ on the fibre $\mathcal E_s$ is nonzero for all $s\in \CP^2$. 
Then, $X=\{q=0\}\subset \CP(\mathcal E)$ defines a quadric surface bundle over $\CP^2$.  
We may also regard $q$ as a symmetric matrix $A=(a_{ij})$, where $a_{ij}$ is a global section of $\mathcal O_{\CP^2}(r_i+r_j+n)$. 
Locally over $\CP^2$, $X$ is given by
\begin{align} \label{eq:X}
\sum_{i,j=0}^3a_{ij} z_iz_j=0 ,
\end{align}
where $z_i$ denotes a local coordinate which trivializes $\mathcal O_{\CP^2}(-r_i)\subset \mathcal E$.  

If $X$ is smooth, its deformation type depends only on the integers $d_i:=2r_i+n$;
 we call any such quadric surface bundle of type $(d_0,d_1,d_2,d_3)$.  
The degeneration locus of $X\longrightarrow \CP^2$ is a plane curve of degree $\sum_i d_i$, which is always even.
If some $d_i$ is negative, then $a_{ii}=0$ and so $X\longrightarrow \CP^2$ admits a section, hence $X$ is rational.
We may thus from now on restrict ourselves to the case $d_i\geq 0$ for all $i$.

\begin{corollary} \label{cor:P(E)} 
Let $d_0,d_1,d_2$ and $d_3$ be non-negative integers of the same parity, and let $X\longrightarrow \CP^2$ be a very general complex projective quadric surface bundle of type $(d_0,d_1,d_2,d_3)$.   
If $\sum_i d_i \neq 6$, then the following holds:
\begin{enumerate} 
\item $X$ is rational if $\sum_i d_i\leq 4$ or if $d_i=d_j=0$ for some $i\neq j$; \label{item:thm:PE:1}
\item $X$ is not stably rational otherwise. \label{item:thm:PE:2}
\end{enumerate}    
\end{corollary} 

As we will see in the proof, the bundles in item (\ref{item:thm:PE:1}) of the above corollary have a rational section, and so already the generic fibre of $X$ over $\CP^2$ is rational.

Up to reordering, the only cases left open by the above corollary are types $(1,1,1,3)$ and $(0,2,2,2)$.
The former corresponds to blow-ups of cubic fourfolds containing a plane (see \cite{ACTP}) and the latter are Verra fourfolds \cite{CKKM,IKKR}, i.e.\ double covers of $\CP^2\times \CP^2$, branched along a hypersurface of bidegree $(2,2)$.
In both exceptions, the degeneration locus of the quadric bundle is a sextic curve in $\CP^2$, and so the associated double cover (cf.\ \cite{APS}) is a K3 surface.

Specializing to $a_{33}=0$ in (\ref{eq:X}) shows that all examples in the above corollary deform to smooth quadric surface bundles with a section, hence to smooth rational fourfolds.

Many quadric surface bundles over $\CP^2$ are birational to fourfolds which arise naturally in projective geometry, cf.\ \cite[Section  3.5]{Sch17}.
For instance, Corollary \ref{cor:P(E)} implies the following 
\begin{enumerate}[(I)]
\item a very general complex hypersurface of bidegree $(d,2)$ in $\CP^2\times \CP^3$ is not stably rational if $d\geq 2$; \label{item:I}
\item a very general complex hypersurface $X\subset \CP^5$ of degree $d+2$ and with multiplicity $d$ along a $2$-plane is not stably rational if $d\geq 2$;\label{item:III}
\item a double cover $X\stackrel{2:1}\longrightarrow \CP^4$, branched along a very general complex hypersurface $Y\subset \CP^4$ of even degree $d+2$ and with multiplicity $d$ along a line is not stably rational if $d\geq 2$.\label{item:II}
\end{enumerate}
The case $d=2$ in items (\ref{item:I}) and (\ref{item:II}) corresponds to the aforementioned results in \cite{HPT} and \cite{HPT2}, respectively. 
For stable rationality properties of smooth hypersurfaces and double covers, see \cite{beauville4,CT-Pirutka,CT-Pirutka2,HPT2,okada,totaro,voisin}; for results on conic bundles, see \cite{AO,artin-mumford,ABGP,beauville4,BB,HKT,voisin}.

In \cite{Sch17}, we studied rationality properties of quadric bundles with arbitrary fibre dimensions.
Our uniform treatment sufficed to prove (\ref{item:I}) and (\ref{item:III}) for $d\geq 5$, and (\ref{item:II}) for $d\geq 8$.  
On the other hand, the results in \cite{Sch17} left open infinitely many cases in Corollary \ref{cor:P(E)}.   
For instance, the types $(1,1,d_2,d_3)$ and $(0,2,d_2,d_3)$ with $d_2\leq 7$ and arbitrary $d_3$ are not covered by \cite{Sch17} and there are more cases which were not accessible, see \cite[Remark 36]{Sch17}.  

Our method applies also to quadric surface bundles over other rational surfaces $S$. 
We treat in this paper the case $S=\CP^1\times \CP^1$ and obtain similar results as those in Corollary \ref{cor:P(E)} above, see Corollaries \ref{cor:P1xP1} and \ref{cor:P1xP1:2} below.

\subsection{Conventions and notations} \label{subsec:conventions}
All schemes are separated.
A variety is an integral scheme of finite type over a field.
A property is said to hold at a very general point of a scheme, if it holds at all closed points in some countable intersection of dense open subsets. 

Let $k$ be an algebraically closed field.
We say that a variety $X$ over a field $L$ specializes (or degenerates) to a variety $Y$ over $k$, if there is a discrete valuation ring $R$ with residue field $k$ and fraction field $F$ with an injection of fields $F\hookrightarrow L$, together with a flat proper morphism $\mathcal X\longrightarrow \Spec R$ of finite type, such that $Y$ is isomorphic to the special fibre $Y\cong \mathcal X\times_R k$ and $X\cong \mathcal X\times_R L$ is isomorphic to a base change of the generic fibre. 
If $\mathcal Y\longrightarrow B$ is a flat  proper morphism of complex varieties with integral fibres, then for any closed points $0,t\in B$ with $t$ very general, the fibre $Y_t$ specializes to $Y_0$ in the above sense, see \cite[Lemma 8]{Sch17}.

A morphism $f:X\longrightarrow Y$ of varieties over a field $k$ is universally $\CH_0$-trivial,  if $f_\ast :\CH_0(X\times L)\stackrel{\cong}\longrightarrow \CH_0(Y\times L)$ is an isomorphism for all field extensions $L$ of $k$.

A  quadric surface bundle is a flat morphism $f:X\longrightarrow S$ between projective varieties such that the generic fibre is a smooth quadric surface; the degeneration locus is given by all $s\in S$ such that $f^{-1}(s)$ is singular.
If $f$ is not assumed flat, then we call $X$ a weak quadric surface bundle over $S$.
Quadric surface bundles over surfaces have been studied in detail in \cite{APS}.

We denote by $\mu_2\subset \mathbb G_m$ the group of second roots of unity. 
If $X$ is a proper variety over a field $k$ of characteristic different from $2$, the unramified cohomology group $H^i_{nr}(k(X)\slash k,\mu_2^{\otimes i})$ is the subgroup of all elements of the Galois cohomology group $H^i(k(X),\mu_2^{\otimes i})$ which have trivial residue at all discrete valuations of rank one on $k(X)$ over $k$, see \cite{CTO}.
This is a stable birational invariant of $X$, see \cite[Proposition 1.2]{CTO}.
If $X$ is smooth and proper over $k$, then $H^i_{nr}(k(X)\slash k,\mu_2^{\otimes i})$ coincides with the subgroup of elements of $H^i(k(X),\mu_2^{\otimes i})$ that have trivial residue at any codimension one point of $X$, see \cite[Theorem 4.1.1]{CT}. 

\section{Second unramified cohomology of quadric surfaces}

Let $K$ be a field of characteristic different from $2$.
It will be convenient to identify the Galois cohomology group $H^i(K,\mu_2^{\otimes i})$  with the \'etale cohomology group $H^i_{\text{\'et}}(\Spec(K),\mu_2^{\otimes i})$.
We also use the identification $H^1(K,\mu_2)\cong K^\ast\slash(K^\ast)^2 $, induced by the Kummer sequence.
For $a,b\in K^\ast$, we denote by $(a,b)\in H^2(K,\mu_2^{\otimes 2})$ the cup product of the classes given by $a$ and $b$.  
If $S$ is a normal variety over a field $k$ and with fraction field $k(S)=K$, then, for any $\alpha \in H^2(K,\mu_2^{\otimes 2})$, the ramification divisor $\ram(\alpha)\subset S$ is given by (the closure of) all codimension one points $x\in S^{(1)}$ with $\del_x^2\alpha \neq 0$.
Here, $\del^2_x:H^{2}(K,\mu_2^{\otimes 2})\longrightarrow H^1(\kappa(x),\mu_2)$ denotes the residue induced by the local ring $\mathcal O_{S,x}\subset K$.

To any non-degenerate quadratic form $q$ over $K$, one associates the discriminant $\discr(q)\in K^\ast\slash(K^\ast)^2$ and the Clifford invariant $\cl(q)\in  H^2(K,\mu_2^{\otimes 2})$.
If $q$ has even dimension, then the discriminant $\discr(q)$ depends only on the quadric hypersurface $Q=\{q=0\}$ and the Clifford invariant satisfies $\cl(\lambda\cdot q)=\cl(q)+(\lambda,\discr(q))$ for all $\lambda\in K^\ast$, see \cite[Chapter 5, (3.16)]{lam}.  
If $Q$ is a surface, then, up to similarity, $q\cong \langle1,-a,-b,abd \rangle$ for some $a,b,d\in K^\ast$.
In this case, $\discr(q)=d$ and $\cl(q)=(-a,-b)+(ab,d)$.  
We will need the following, see \cite{arason} and \cite[Corollary 8]{KRS}.

\begin{theorem}
\label{thm:Arason}
Let $K$ be a field with $\operatorname{char}(K)\neq 2$, and let $f:Q\longrightarrow \Spec K$ be a smooth projective quadric surface over $K$.
Denote by $d\in K^\ast\slash(K^\ast)^2$ the discriminant of $Q$ and by $\beta\in H^2(K,\mu_2^{\otimes 2})$ the Clifford invariant of some quadratic form $q$ with $Q=\{q=0\}$.
Then, 
$$
f^\ast: H^2(K,\mu_2^{\otimes 2})\longrightarrow H^2_{nr}(K(Q)\slash K,\mu_2^{\otimes 2})
$$ 
is an isomorphism if $d$ is nontrivial.
If $d\in (K^\ast)^2$, then $\ker(f^\ast)=\{1,\beta\}$. 
\end{theorem}

Pirutka computed the unramified cohomology group $H^2_{nr}(K(Q)\slash \C,\mu_2^{\otimes 2})$ of a smooth quadric surface $Q$ with nonzero discriminant over the function field of a smooth complex surface, see \cite[Theorem 3.17]{Pirutka}. 
The following reflects one half of her result.

\begin{theorem}[Pirutka] \label{thm:Pirutka} 
Let $f:Q\longrightarrow \Spec K$ be a smooth projective quadric surface over the function field $K$ of some smooth surface $S$ over $\C$. 
Let $d\in K^\ast \slash (K^{\ast})^2$ denote the discriminant and $\beta \in H^2(K,\mu_2^{\otimes 2})$ the Clifford invariant of some quadratic form $q$ with $Q=\{q=0\}$. 
If for some $\alpha\in H^2(K,\mu_2^{\otimes 2})$ the pullback $f^\ast (\alpha) \in H^2_{nr}(K(Q)\slash K,\mu_2^{\otimes 2})$  is unramified over $\C$, then the following holds: 
\begin{enumerate}
\item[($\ast$)] If the residue $\del^2_x\alpha$ at some codimension one point $x\in S^{(1)}$ is nonzero, then 
\begin{enumerate}
\item $\del^2_x\alpha =\del^2_x \beta$;  
\item \label{item:Pirutka:d} 
 $d$ becomes a square in the fraction field  of the completion $\widehat {\mathcal O_{S,x}}$. 
\end{enumerate}  
\end{enumerate} 
\end{theorem}

\begin{proof}  
The condition on $d$ is by Hensel's lemma equivalent to asking that, up to multiplication by a square,  $d$ is a unit in $\mathcal O_{S,x}$ whose image in $\kappa(x)$ is a square.
The theorem follows therefore from \cite[Section 3.6.2]{Pirutka}.
(In \cite[Theorem 3.17]{Pirutka}, the assumption that $d$ is not a square is only used to invoke bijectivity of $f^\ast$ via Theorem \ref{thm:Arason}; the assumption that $\ram(\beta)$ is a simple normal crossing divisor on $S$ is only used in \cite[Section 3.6.1]{Pirutka}.)
\end{proof}

\begin{remark} \label{rem:Pirutka} 
Up to replacing $S$ by some blow-up, one can always assume that $\ram(\beta)$ is a simple normal crossing divisor on $S$.
Under this assumption, Pirutka's analysis in \cite[Section 3.6.1]{Pirutka} shows that the following converse of the above theorem is also true: 
if $\alpha\in H^2(K,\mu_2^{\otimes 2})$ is such that condition ($\ast$) holds, then $f^\ast \alpha\in H^2_{nr}(K(Q)\slash K,\mu_2^{\otimes 2})$ is unramified over $\C$; non-triviality can be checked via Theorem \ref{thm:Arason}.
\end{remark}

Pirutka's result \cite[Theorem 3.17]{Pirutka} applies to the following important example, due to Hassett, Pirutka and Tschinkel \cite[Proposition 11]{HPT}; for a reinterpretation in terms of conic bundles, see \cite{ABGP}.

\begin{proposition}[Hassett--Pirutka--Tschinkel] \label{prop:HPT}
Let $K=\C(x,y)$ be the function field of $\CP^2$ and consider the quadratic form $q=\langle y,x,xy,F(x,y,1)\rangle$ over $K$, where 
$$
F(x,y,z)=x^2+y^2+z^2-2(xy+xz+yz).
$$ 
If $f:Q\longrightarrow \Spec K$ denotes the corresponding projective quadric surface over $K$, then 
$$
0\neq f^\ast ((x,y)) \in H^2_{nr}(K(Q)\slash \C, \mu_2^{\otimes 2}) .
$$
\end{proposition} 

\section{A vanishing result} 
The following general vanishing result is the key ingredient of this paper. 

\begin{proposition} \label{prop:alpha}
Let $Y$ be a smooth complex projective variety and let $S$ be a smooth complex projective surface.
Let $f:Y\dashrightarrow S$ be a dominant rational map whose generic fibre $Y_\eta$ is stably birational to a smooth quadric surface $Q$ over $K=\C(S)$.  
Suppose that there is some $\alpha \in H^2(K,\mu_2^{\otimes 2})$, such that $\alpha':=f^\ast \alpha\in H_{nr}^2(K(Y_\eta)\slash K,\mu_2^{\otimes 2})$ is unramified over $\C$. 
Then, for any prime divisor $E\subset Y$ which does not dominate $S$, the restriction  of $\alpha'$ to $E$ vanishes: 
$$
\alpha'|_E=0\in H^2(\C(E),\mu_2^{\otimes 2}). 
$$
\end{proposition}

\begin{proof} 
Since unramified cohomology is a functorial stable birational invariant (see \cite{CTO}), we may up to replacing $Y$ by $Y\times \CP^m$ assume that $Y_\eta$ is birational to $Q\times \CP^r_K$ for some $r\geq 0$.
This birational map induces a dominant rational map $Y_{\eta}\dashrightarrow Q$. 

Since $Y$ is smooth, $f$ is defined at the generic point $y$ of $E$.  
By \cite[Propositions 1.4 and 1.7]{merkurjev} (cf.\ \cite[Section 5]{Sch17}), we may up to replacing $S$ by a different smooth projective model assume that the image $x:=f(y)\in S^{(1)}$ is a codimension one point on $S$. 
Consider the local ring $A:=\mathcal O_{S,x}$ and let $\widehat A$ be its completion with field of fractions $\widehat K:=\Frac(\widehat A)$. 
The local ring $B:= \mathcal O_{ Y,y}$ contains $A$.
We let $\widehat B$ be the completion of $B$ and $\widehat L:=\Frac(\widehat B)$ be its field of fractions. 
Since $Y_{\eta}\dashrightarrow Q$ is dominant, inclusion of fields induces the following sequence
\begin{align}\label{seq:Khat}
H^2(K,\mu_2^{\otimes 2})\stackrel{\varphi_1}\longrightarrow H^2(\widehat K,\mu_2^{\otimes 2}) \stackrel{\varphi_2}\longrightarrow H^2(\widehat K(Q),\mu_2^{\otimes 2}) \stackrel{\varphi_3}\longrightarrow H^2(\widehat L,\mu_2^{\otimes 2}) .
\end{align} 

\begin{lemma}\label{lem:milne}
If some $\gamma\in H^2(K,\mu_2^{\otimes 2})$ satisfies $\del_x^2\gamma=0$, then $ \varphi_1(\gamma)=0\in H^2(\widehat K,\mu_2^{\otimes 2}) $. 
\end{lemma}
\begin{proof} 
Since $\del_x^2\gamma=0$, the image of $\gamma$ in $H^2(\widehat K,\mu_2^{\otimes 2})$ is contained in $H^2_{\text{\'et}}(\Spec \widehat A,\mu_2^{\otimes 2})\subset H^2(\widehat K,\mu_2^{\otimes 2})$, see \cite[\S 3.3 and \S 3.8]{CT}.
It thus suffices to show that $H^2_{\text{\'et}}(\Spec \widehat A,\mu_2^{\otimes 2})$ vanishes. 
Since $\widehat A$ is a henselian local ring,  
restriction to the closed point gives an isomorphism
$
H^2_{\text{\'et}}(\Spec \widehat A,\mu_2^{\otimes 2})\cong H^2(\kappa(x),\mu_2^{\otimes 2})
$, see \cite[Corollary VI.2.7]{milne}.
By Tsen's theorem, $H^2(\kappa(x),\mu_2^{\otimes 2})=0$.
This concludes the lemma.
\end{proof}

Since $f^\ast \alpha$ is unramified, we know that 
\begin{align} \label{eq:phiialpha}
\varphi_3\circ \varphi_2\circ \varphi_1(\alpha) \in H^2_{\text{\'et}}(\Spec \widehat B,\mu_2^{\otimes 2})\subset H^2(\widehat L,\mu_2^{\otimes 2}),
\end{align}
see \cite[\S 3.3 and \S 3.8]{CT} and the compatibility of the residue map illustrated in \cite[p.\ 143]{CTO}.  
We aim to show that this class vanishes, which is enough to conclude the proposition, because $\alpha'|_E$ is obtained as the restriction of the above class to the closed point $\Spec \C(E)$. 

In order to show that (\ref{eq:phiialpha}) vanishes, we choose some quadratic form $q$ with $Q=\{q=0\}$ and denote by $d\in K^{\ast}\slash(K^{\ast})^2$ and $\beta\in H^2(K,\mu_2^{\otimes 2})$ the discriminant and the Clifford invariant of $q$, respectively.
If $\del_x^2\alpha=0$, then (\ref{eq:phiialpha}) vanishes by Lemma \ref{lem:milne}.
If $\del_x^2\alpha\neq 0$, then $\del_x^2(\alpha-\beta)=0$ by Theorem \ref{thm:Pirutka}, because $Y_\eta$ is stably birational to $Q$ and unramified cohomology is a stable birational invariant.
By Lemma \ref{lem:milne}, it then suffices to show that $\beta$ maps to zero via (\ref{seq:Khat}). 
By Theorem \ref{thm:Pirutka}, $d$ becomes a square in $\widehat K$, and so the latter follows from Theorem \ref{thm:Arason}, applied to $\varphi_2$ in (\ref{seq:Khat}). 
This concludes the proof of the proposition.
\end{proof}

\section{Proof of Theorem \ref{thm:quadric}}

The following is a generalization of Theorem \ref{thm:quadric}, stated in the introduction.
For what it exactly means that a variety specializes to another variety, see Section \ref{subsec:conventions} above . 

\begin{theorem} \label{thm:quadric:2}
Let $X$ be a proper variety which specializes to a complex projective variety $Y$.
Suppose that there is a dominant rational map  $f:Y\dashrightarrow \CP^2$ with the following properties: 
\begin{enumerate} [\ \ \ (a)]
\item some Zariski open and dense subset $U\subset Y$ admits a universally $\CH_0$-trivial resolution of singularities $\widetilde U\longrightarrow U$ such that the induced rational map $\widetilde U\dashrightarrow \CP^2$ is a morphism whose generic fibre is proper over $\C(\CP^2)$. 
 \label{item:tildeU}
\item the generic fibre $Y_{\eta}$ of $f$ is stably birational to a smooth projective quadric surface $g:Q\longrightarrow  \Spec K$ over $K=\C(\CP^2)$, such that
 \label{item:Yeta=Q}
there is a class $\alpha \in H^2(K,\mu_2^{\otimes 2})$ whose pullback $g^\ast \alpha$ is nontrivial and unramified over $\C$:  
$$
0\neq g^\ast \alpha \in H^2_{nr}(K(Q)\slash \C,\mu_2^{\otimes 2})=H^2_{nr}(\C(Y)\slash \C,\mu_2^{\otimes 2}) .
$$
\end{enumerate} 
Then, no resolution of singularities of $X$ admits an integral decomposition of the diagonal.
In particular, $X$ is not stably rational.
\end{theorem}

\begin{proof} 
Since $g^\ast \alpha\neq 0$ is unramified over $\C$ and unramified cohomology is a stable birational invariant, $\alpha':=f^\ast \alpha\in H^2(\C(Y),\mu_2^{\otimes 2})$ is a nontrivial class which is unramified over $\C$. 
By Hironaka's theorem, there exists a resolution of singularities  $\tau:\widetilde Y\longrightarrow Y$, such that $\tau^{-1}(U)$ identifies with the resolution of singularities $\widetilde U$ of $U$ given in (\ref{item:tildeU}), and such that $E:=\widetilde Y\setminus \widetilde U$ is a simple normal crossing divisor in $\widetilde Y$.
Our assumption on $\widetilde U$ then implies that $\tau^{-1}(U)\longrightarrow U$ is universally $\CH_0$-trivial.
Moreover, each component $E_i$ of $E$ is smooth and does not dominate $\CP^2$.
Therefore, Proposition \ref{prop:alpha} implies that the nontrivial class $\alpha'$ restricts to zero on $E_i$ for all $i$ and so 
Theorem \ref{thm:quadric:2} follows from the new key technique in \cite[Section 4]{Sch17}. 
\end{proof}

\begin{proof}[Proof of Theorem \ref{thm:quadric}]  
Condition (\ref{item:Q}) in Theorem \ref{thm:quadric} implies condition (\ref{item:tildeU}) in Theorem \ref{thm:quadric:2} with $\widetilde U=U$.
By Theorem \ref{thm:Arason}, conditions (\ref{item:Q}), (\ref{item:d}) and (\ref{item:H_nr}) in Theorem \ref{thm:quadric} imply condition (\ref{item:Yeta=Q}) in Theorem \ref{thm:quadric:2}. 
Theorem \ref{thm:quadric} follows therefore from Theorem \ref{thm:quadric:2}. 
\end{proof}

\section{Applications}

\subsection{Quadric surface bundles over $\CP^2$} \label{subsec:Application:P^2}
If the symmetric matrix $A=(a_{ij})$ in (\ref{eq:X}) is of diagonal form, i.e.\ $a_{ij}=0$ for all $i\neq j$, then we say that the corresponding quadric surface bundle $X$ is given by the quadratic form $q=\langle a_{00},\dots ,a_{33}\rangle$.
The condition that $X$ is flat over $\CP^2$ means that the $a_{ii}$ have no common zero.
If the homogeneous polynomials $a_{ii}$ degenerate and acquire common zeros, then the same formula still defines a weak quadric bundle as long as the $a_{ii}$ are nonzero and have no common factor.
We will use such degenerations in the proofs below.

\begin{proof}[Proof of Corollary \ref{cor:P(E)}]
In the notation of (\ref{eq:X}), let $A=(a_{ij})_{0\leq i,j\leq 3}$ be the symmetric matrix which corresponds to the very general quadric surface bundle $X$ of type $(d_0,d_1,d_2,d_3)$ over $\CP^2$.
We may without loss of generality assume $0\leq d_0\leq d_1\leq d_2\leq d_3$.  
If $d_1=0$, then also $d_0=0$ and $a_{ij}\in \C$ is constant for $i,j\in \{0,1\}$.
The quadric $\{a_{00}z_0^2+2a_{01}z_0z_1+a_{11}z_1^2=0\}$ has thus a point over $\C$ and so $X\longrightarrow \CP^2$ has a section. 
Hence $X$ is rational. 
If $d_i=1$ for all $i$, then $X$ is a hypersurface of bidegree $(1,2)$ in $\CP^2\times \CP^3$ and so projection to the second factor shows that $X$ is rational.
Since the $d_i$ have all the same parity, this shows that $X$ is rational if $\sum d_i\leq 4$ or $d_1=0$.

The case $d_i=2$ for all $i$ is due to \cite{HPT}; a quick proof follows from \cite[Proposition 11]{HPT} (= Proposition \ref{prop:HPT} above) and Theorem \ref{thm:quadric}.

It remains to deal with the case where 
$\sum_id_i\geq 8$, $d_1\geq 1$ and $d_3\geq 3$.
Recall that all $d_i$ are either even or odd.
Consider the  weak quadric surface bundle $Y_i:=\{q_i=0\}\subset \CP(\mathcal E)$ of type $(d_0,d_1,d_2,d_3)$, given by the diagonal forms
\begin{align*}
q_1&:=\langle z^{d_0},x^{d_1}, xyz^{d_2-2},yz^{d_3-3}F(x,y,z) \rangle ,\\
q_2&:=\langle z^{d_0},xz^{d_1-1},x^{d_2-1}y,yz^{d_3-3}F(x,y,z) \rangle ,\\
q_3&:= \langle z^{d_0},x^{d_1},yz^{d_2-1},xyz^{d_3-4}F(x,y,z) \rangle ,
\end{align*}
where $F$ is the quadratic polynomial from Proposition \ref{prop:HPT}.

Note that $Y_i$ is integral, because the entries in the diagonal form are coprime. 
Consider the natural projection $Y_i\longrightarrow \CP^2$.
The generic fibre is a smooth quadric surface $Q_i$ over $K=\C(\CP^2)$.
Setting $z=1$ shows that  $Q_1$ is given by the quadratic form  $q_1'=\langle1,x^{d_1},xy,yF(x,y,1)\rangle$,  $Q_2$ is given by $q_2'=\langle1,x,x^{d_2-1}y,yF(x,y,1)\rangle$, and $Q_3$ is given by $q_3'=\langle1,x^{d_1},y,xyF(x,y,1)\rangle$.

If $d_0$ is even, then so is $d_2$.
Multiplying through by $y$, absorbing squares and reordering the entries shows thus in this case that  $q_2'$ is similar to the quadratic form $q=\langle y,x,xy,F(x,y,1)\rangle$ from Proposition \ref{prop:HPT}.
If $d_0$ is odd, then so is $d_1$ and so $q_1'$ is isomorphic to $\langle 1,x,xy,yF(x,y,1)\rangle $ and $q_3'$ is isomorphic to $\langle 1,x,y,xyF(x,y,1)\rangle $.
Again, $q_1'$ and $q_3'$ are both similar to $q$.  
Hence, $H^2_{nr}(K(Q_i)\slash \C,\mu_2^{\otimes 2})\neq 0$ for $i\equiv d_0 \mod 2$ by \cite[Proposition 11]{HPT} (= Proposition \ref{prop:HPT} above). 

Since $d_1,d_2\geq 1$ and $d_3\geq 3$, the very general quadric surface bundle $X\subset \CP(\mathcal E)$ as in Corollary \ref{cor:P(E)} degenerates to $Y_2$.
If $d_0$ is odd, $X$ also degenerates to $Y_1$ or $Y_3$, depending on whether $d_2\geq 3$ or $d_2=1$.
Depending on the parity of $d_0$ and the size of $d_2$, we can choose one of the three degenerations together with Theorem \ref{thm:quadric} (or \ref{thm:quadric:2}) to conclude. 
\end{proof}

\begin{remark}
Pirutka informed me that for any total degree $d:=\sum_i d_i\geq 8$, one can reprove some cases of Corollary \ref{cor:P(E)} via degenerations to similar quadric surface bundles as in \cite{HPT}, for which \cite[Theorem 3.17]{Pirutka} applies, and for which one can compute universally $\CH_0$-trivial resolutions explicitly, see \cite{ABP}.
\end{remark}

\subsection{Quadric surface bundles over $\CP^1\times \CP^1$}

As a second example where Theorem \ref{thm:quadric} applies, we consider quadric surface bundles $X$ over $\CP^1\times \CP^1$ that are given by a line bundle valued quadratic form $q:\mathcal E\longrightarrow \mathcal O(m,n)$, where  $\mathcal E=\bigoplus_{i=0}^3 \mathcal O(-p_i,-q_i)$ is split. 
Locally, $X:=\{q=0\}\subset \CP(\mathcal E)$ is given by (\ref{eq:X}) where $a_{ij}$ is a global section 
of $\mathcal O(p_i+p_j+m,q_i+q_j+n)$. 
If $a_{ij}=0$ for $i\neq j$, we say that $X$ is given by the quadratic form $q=\langle a_{00},\dots ,a_{33}\rangle$.
If the $a_{ii}$ degenerate and acquire common zeros, then the same formulas still define a hypersurface in $\CP(\mathcal E)$ which is a weak quadric surface bundle over $\CP^2$ as long as the $a_{ii}$ are nonzero and have no common factor.
The deformation type of $X$ depends only on the integers $d_i:=m+2p_i$ and $e_i:=n+2q_i$, and we call $(d_i,e_i)_{0\leq i\leq 3}$ the type of $X$. 
Note that the $d_i$, as well as the $e_i$ have the same parity for all $i$. 
We say that the type $(d_i,e_i)_{0\leq i\leq 3}$ is lexicographically ordered, if $d_i< d_{i+1}$ or $d_i=d_{i+1}$ and $e_i\leq e_{i+1}$.

\begin{corollary} \label{cor:P1xP1}  
Let  $X\longrightarrow \CP^1\times \CP^1$ be a very general quadric surface bundle of lexicographically ordered type $(d_i,e_i)_{0\leq i\leq 3}$, with $d_i,e_i\geq 0$ and $d_3,e_3\geq 3$. 
Then, 
\begin{enumerate}
\item $X$ is rational if $d_2=0$, $d_1=e_1=e_0=0$ or $e_0=e_1=e_2=0$;
\item $X$ is not stably rational otherwise.
\end{enumerate}
\end{corollary}

All examples in Corollary \ref{cor:P1xP1} deform to smooth rational varieties of dimension four, see for instance \cite[Section 3.5]{Sch17}.
The condition $d_3,e_3\geq 3$ in the above theorem could be replaced by a weaker but more complicated assumption; we collect in Corollary \ref{cor:P1xP1:2} below the remaining cases where our method works. 

\begin{proof} [Proof of Corollary \ref{cor:P1xP1}]
Let $A=(a_{ij})_{0\leq i,j\leq 3}$ be a symmetric matrix, where $a_{ij}$ is a very general global section of $\mathcal O_{\CP^1\times \CP^1}(p_i+p_j+m,q_i+q_j+n)$ and consider the corresponding quadric surface bundle $X$ over $\CP^1\times \CP^1$.
Here the integers $d_i:=2p_i+m$ and $e_i:=2q_i+n$ are assumed to satisfy the assumptions of Corollary \ref{cor:P1xP1}, i.e.\ $(d_i,e_i)_{0\leq i\leq 3}$ is lexicographically ordered with $d_i,e_i\geq 0$ and $d_3,e_{3}\geq 3$.

If $d_1=e_1=e_0=0$, then $(a_{ij})_{0\leq i,j\leq 1}$ is a constant matrix and so $X$ has a section. 
If $d_2=0$, then $(a_{ij})_{0\leq i,j\leq 2}$ is a matrix of polynomials which are constant along the first factor. 
Since any conic bundle over $\CP^1$ has a section, $X$ admits also a section.
If $e_0=e_1=e_2=0$, then $(a_{ij})_{0\leq i,j\leq 2}$ is a matrix of polynomials, constant along the second factor, and so $X$ has a section as before.
Since $X$ is general and $d_i,e_i\geq 0$, the generic fibre of $X$ over $\CP^1\times \CP^1$ is a smooth quadric surface and so $X$ is rational in each of the above cases. 

The case where $(e_0,e_1,e_2)\neq(0,0,0)$, $(d_1,e_0,e_1)\neq (0,0,0)$ and $d_2\neq 0$ is similar to the proof of Corollary \ref{cor:P(E)}.  
The main point is that we can always degenerate $X$ to weak quadric surface bundle $Y$ over $\CP^1\times \CP^1$ whose generic fibre is isomorphic to the example in Proposition \ref{prop:HPT}.
To find such a degeneration, we consider coordinates $x_0,x_1$ (and $y_0,y_1$) on the first (respectively second) factor of $\CP^1\times \CP^1$ and consider the bidegree $(2,2)$ polynomial 
\begin{align} \label{eq:h}
h:=
x_1^2y_0^{2}+x_0^{2}y_1^2+x_0^{2}y_0^{2}
-2(x_1y_1x_0y_0+x_1x_0y_0^2+y_1y_0x_0^2) .
\end{align}
We then start with the quadratic form $q=\langle 1,y_1,x_1,x_1y_1h \rangle$.
Putting $x_0=y_0=1$ shows that the corresponding quadric surface over $K=\C(\CP^1\times \CP^1)$ is isomorphic to the one in Proposition \ref{prop:HPT}. 
The point is that the isomorphism type of this quadric surface does not change if we perform any of the following operations to the quadratic form $q$:
\begin{itemize} 
\item multiply some entries with even powers of $x_1$ and $y_1$;
\item multiply some entries with arbitrary powers of $x_0$ and $y_0$;
\item reorder the entries of the quadratic form.
\end{itemize} 
Our aim is to produce a quadratic form of given type $(e_i,d_i)_{0\leq i\leq 3}$ whose entries are coprime, since the latter guarantees that the associated quadratic form defines a weak quadric surface bundle $Y$ over $\CP^1\times \CP^1$.
Once this is achieved, Corollary \ref{cor:P1xP1} will follow from Proposition \ref{prop:HPT} 
and Theorem \ref{thm:quadric}. 

By assumption, $d_2\geq 1$, and if $e_0=e_1=0$, then $d_1\geq 1$ and $e_2\geq 1$.
This leads to Cases A, B and C below.
We divide into further subcases and provide each time a quadratic form (produced via the above process) with the properties we want. 
Recall that the $d_i$, as well as the $e_i$, have the same parity.

\textbf{Case A:}
$e_1\geq 1$.

\emph{Subcase A.1.} 
If $d_0$ and $e_0$ are even, then we take
$$
\langle x_1^{d_0}y_1^{e_0},x_0^{d_1}y_0^{e_1-1}y_1,x_0^{d_2-1}x_1y_0^{e_2},x_0^{d_3-3}y_0^{e_3-3}x_1y_1h  \rangle .
$$

\emph{Subcase A.2.} 
If $d_0$ is odd and $e_0$ is even, then we take
$$
\langle x_0^{d_0}y_1^{e_0},x_0^{d_1}y_0^{e_1-1}y_1,x_1^{d_2}y_0^{e_2},x_0^{d_3-3}y_0^{e_3-3}x_1y_1h  \rangle .
$$

\emph{Subcase A.3.} 
If $d_0$ is even and $e_0$ is odd, then we take
$$
\langle x_1^{d_0}y_0^{e_0},x_0^{d_1}y_1^{e_1},x_0^{d_2-1}x_1y_0^{e_2},x_0^{d_3-3}y_0^{e_3-3}x_1y_1h  \rangle .
$$

\emph{Subcase A.4.} 
If $d_0$ and $e_0$ are odd, then we take
$$
\langle x_0^{d_0}y_0^{e_0},x_0^{d_1}y_1^{e_1},x_1^{d_2}y_0^{e_2},x_0^{d_3-3}y_0^{e_3-3}x_1y_1h  \rangle .
$$

\textbf{Case B:} $e_0\geq 1$ and $e_1=0$ (hence $e_i$ is even for all $i$).

\emph{Subcase B.1.} 
If $d_0$ is even, then we take
$$
\langle x_1^{d_0}y_0^{e_0-1}y_1,x_0^{d_1},x_0^{d_2-1}x_1y_0^{e_2},x_0^{d_3-3}y_0^{e_3-3}x_1y_1h  \rangle .
$$

\emph{Subcase B.2.} 
If $d_0$ is odd, then we take
$$
\langle x_0^{d_0}y_0^{e_0-1}y_1,x_0^{d_1},x_1^{d_2}y_0^{e_2},x_0^{d_3-3}y_0^{e_3-3}x_1y_1h  \rangle .
$$

\textbf{Case C:}
$d_1,e_2\geq 1$ and $e_0=e_1=0$ (hence $e_i$ is even for all $i$).

\emph{Subcase C.1.} 
If $d_0$ is even, then we take
$$
\langle x_1^{d_0},x_0^{d_1-1}x_1,x_0^{d_2}y_0^{e_2-1}y_1,x_0^{d_3-3}y_0^{e_3-3}x_1y_1h  \rangle .
$$

\emph{Subcase C.2.} 
If $d_0$ is odd, then we take
$$
\langle x_0^{d_0},x_1^{d_1},x_0^{d_2}y_0^{e_2-1}y_1,x_0^{d_3-3}y_0^{e_3-3}x_1y_1h  \rangle .
$$

In each of the above cases, putting $x_0=y_0=1$ and reordering the factors if necessary shows that the corresponding weak quadric surface bundle $Y$ over $\CP^1\times \CP^1$ has generic fibre which is isomorphic to $\langle1,y_1,x_1,x_1y_1F(x_1,y_1,1)\rangle$.
Corollary \ref{cor:P1xP1} follows therefore from \cite[Proposition 11]{HPT} (see Proposition \ref{prop:HPT} above) and Theorem \ref{thm:quadric}. 
\end{proof}

\begin{corollary} \label{cor:P1xP1:2}
Let $(d_i,e_i)_{0\leq i \leq 3}$ be a lexicographically ordered tuple of pairs of non-negative integers with $d_i+d_j$ and $e_i+e_j$ even for all $i,j$.
Suppose that one of the following holds:
\begin{enumerate}
\item $d_1\geq 1$, $d_3\geq 2$, $e_1+e_2\geq 1$ and $e_3\geq 3$; \label{item1:P1xP1:2}
\item $d_1\geq 1$, $d_3\geq 2$, $e_0\geq 1$, $e_1+e_2\geq 1$ and $e_2\geq 2$. \label{item2:P1xP1:2}
\end{enumerate}
Then a very general complex projective quadric surface bundle $X$ over $\CP^1\times \CP^1$ of type $(d_i,e_i)_{0\leq i \leq 3}$ is not stably rational.
\end{corollary}

\begin{proof}
We start with the quadratic forms $q_1:=\langle 1,x_1,x_1y_1,y_1h \rangle$ and $q_2:=\langle y_1,x_1,x_1y_1,h \rangle$, where $h$ is as in (\ref{eq:h}). 
If condition (\ref{item1:P1xP1:2}) holds, then we can use $q_1$ and if (\ref{item2:P1xP1:2}) holds, then we can use $q_2$ to obtain via the procedure explained in the proof of Corollary \ref{cor:P1xP1}, a quadratic form of type $(d_i,e_i)_{0\leq i \leq 3}$ whose coefficients are coprime. 
This yields a special fibre to which Theorem \ref{thm:quadric} applies.
The details are similar as in the proof of Corollary \ref{cor:P1xP1} and we leave them to the reader.
\end{proof}

\section*{Acknowledgements} 
I am very grateful to O.\ Benoist for useful correspondence and comments on a previous version. 
I am very grateful to the referees and to J.-L.\ Colliot-Th\'el\`ene for many suggestions which improved the exposition. 
Thanks to A.\ Pirutka and C.\ Camere for useful comments.  
The results of this article were conceived when the author was member of the SFB/TR 45. 


\end{document}